\documentclass{article}
\usepackage{hyperref}
\usepackage[american]{babel}
\usepackage{amsfonts,amsmath,amssymb,epsf,epsfig}

\usepackage{amsthm}

\usepackage{xypic}
\usepackage{xy}
\xyoption{all}
\input{xypic}
\newdir{ >}{{}*!/-8pt/\dir{>}}

\newcommand{\id}{{\rm id}}

\def\R{\mathbb{R}}

\def\T{\mathbb{T}}

\def\Ad{{\rm Ad}}     
  
\def\pr{$\bf{Proof.}$\quad}
\def\fin{\hfill$\square$\\}

\def\Bis{{\rm Bis}}
\def\Tau{{\mathcal T}}

\newtheorem{theo}{Theorem}[section]
\newtheorem{Defi}[theo]{{\bf Definition}}
\newenvironment{defi}{\begin{Defi} \normalfont}{\end{Defi}}

\newtheorem{Prop}[theo]{{\bf Proposition}}
\newenvironment{prop}{\begin{Prop} \normalfont}{\end{Prop}}
\newtheorem{Cor}[theo]{{\bf Corollary}}
\newenvironment{cor}{\begin{Cor} \normalfont}{\end{Cor}}
\newtheorem{Lem}[theo]{{\bf Lemma}}
\newenvironment{lem}{\begin{Lem} \normalfont}{\end{Lem}}
\newtheorem{pro}[theo]{Problem}
\newtheorem{Exa}[theo]{{\bf Example}}

\newtheorem{Rem}[theo]{{\bf Remark}}
\newenvironment{rem}{\begin{Rem} \normalfont}{\end{Rem}}

\newenvironment{prf}{\begin{proof}}{\end{proof}}

\begin{document}
\title{Lie rackoids} 

\author{Camille Laurent-Gengoux\\
        Universit\'e de Lorraine (Metz)\\
\and Friedrich Wagemann\\
     Universit\'e de Nantes}

\maketitle

\begin{abstract}
We define a new differential geometric structure, called Lie rackoid. 
It relates to Leibniz algebroids exactly as Lie groupoids relate to
Lie algebroids. Its main ingredient is a selfdistributive product
on the manifold of bisections of a smooth precategory. 

We show that the tangent algebroid of a Lie rackoid is a Leibniz algebroid
and that Lie groupoids gives rise via conjugation to a Lie rackoid. Our main 
objective are large classes of examples, including a Lie rackoid integrating 
the Dorfman bracket without the cocycle term of the standard Courant algebroid.   
\end{abstract}

\section*{Introduction}

There is a growing literature dealing with what is called indifferently Leibniz algebra 
or Loday algebra, see for example \cite{Lod}, and its infinite dimensional counterpart, 
namely Leibniz or Loday algebroids, see e.g. \cite{ILMP}, \cite{Wad}.  
We shall prefer here to use the name Leibniz algebra (following J.-L. Loday), 
rather than Loday algebra, although Loday is certainly the main promotor of the 
theory\footnote{The concept of 
a Leibniz algebra appears in a paper by Blokh in 1965.}. Recall that a Leibniz or Loday 
algebra is a vector space ${\mathfrak h}$ equipped with a bilinear bracket
that satisfies a Jacobi identity (meaning that the the left adjoint is a derivation of 
the bracket) although it is not necessarily skew-symmetric. In an equation:
$$[X,[Y,Z]]\,=\,[[X,Y],Z]+[Y,[X,Z]]$$
for all elements $X,Y,Z\in{\mathfrak h}$.
A Leibniz algebroid is then a vector bundle on some manifold $M$ together with an anchor 
map such that the space of sections carries the structure of a Leibniz algebra.   
A subtlety with the definition of Leibniz algebroid is that we only have, in general, 
an anchor map for the left adjoint action. Observe that under light conditions a Leibniz 
bracket which is a derivation in both of its arguments must be antisymmetric, see 
\cite{GraMar}.   

Our main goal in this article is to investigate the integrated or group-version of 
a Leibniz algebroid. We call this new structure {\it Lie rackoids}, because the tangent 
space at 1 of a Lie rack is a Leibniz algebra (see \cite{Kin}) generalizing the tangent 
Lie algebra of a Lie group. It is therefore natural to call the -oid version Lie rackoid. 
Our definition of a Lie rackoid draws on the conjugation Lie rackoid underlying a Lie 
groupoid. Indeed, the set of bisections in a Lie groupoid is a major actor for defining
a conjugation in a Lie groupoid. We define, roughly speaking, a Lie rackoid structure 
as being a selfdistributive operation, i.e. for all $x,y,z$
$$x\rhd(y\rhd z)\,=\,(x\rhd y)\rhd(x\rhd z),$$
on the set of bisections of some precategory which 
also acts on all arrows in the precategory. 

It is thus clear that one big class of examples for Lie rackoids are those underlying 
a Lie groupoid (Proposition 5.1). Another expected result is that the tangent structure
of a Lie rackoid is a Leibniz algebroid (Theorem 5.3). More classes of examples are 
obtaind from rackoid structures on those precategories where source map and target map 
coincide (bundles of Lie racks, see Section 4), or from the augmented rack construction,
suitably transposed in the present context (see Section 6.2). Our main example (see 
Proposition 6.2) is the 
Lie rackoid integrating the hemisemidirect product Leibniz algebroid made from the action 
of vector fields on 1-forms with the bracket
$$[X+\alpha,Y+\beta] = [X,Y]+ {\mathcal L}_X \beta$$
for all $X,Y \in {\mathcal X}(M), \alpha, \beta \in \Omega^1(M)$.
It is clear that this Leibniz algebroid is, up to a cocycle, the Dorfman bracket of
the standard Courant 
algebroid. Our search for the correct definition of a Lie rackoid was largely motivated 
by the integration of Courant algebroids, which was achieved in the framework of 
graded geometry in \cite{MetTan}, \cite{LiBSev}, \cite{SheZhu}, but which we wanted to 
explore using ordinary differential geometry. Namely, in the graded geometry framework, 
it is not clear how Dirac structures give rise to Lie subgroupoids of the integrated 
object. We plan to clarify this using a Lie rackoid integrating the standard Courant 
algebroid in a follow-up article.

In order to explain our concepts in a down-to-earth manner, we start by explaining the 
concept of a (plain) rackoid (Definition  \ref{definition_rackoid}), not involving 
smoothness. Augmented rackoids
are taken care of in Section 2.4 - the main construction mechanism for rackoids in our 
article. Then we pass to Lie rackoids in Section 3. The main idea is to give the set of 
bisections of a smooth precategory an infinite-dimensional (Fr\'echet-) manifold 
structure, inspired by work of Schmeding-Wockel \cite{SchWoc}. A Lie rackoid is then a
smooth selfdistributive structure on this manifold of bisections (Definition 3.5). 
The rest of the article discusses the above mentioned structure results and classes of 
examples.       
   
\noindent{\bf Acknowledgements:}
FW thanks Universit\'e Lorraine for financing several visits to Metz where this 
research has been carried out.  

\section{Bisections in a Lie groupoid}

Let $\xymatrix{\Gamma\ar@<2pt>[r]^{s}\ar@<-2pt>[r]_{t} & M}$ be a groupoid on a base 
manifold $M$ with source $s$ and target $t$. 
Notice that we use composition conventions which are the opposite of those in \cite{Mack}: 
our groupoid multiplications are read from left to right. This means in particular that 
two elements $x,y \in \Gamma$ are compatible (i.e. their product $xy$ is defined) if and 
only if $t(x)=s(y)$.

For two subsets $X,Y \subset \Gamma$, we shall denote by $X \star Y$ the subset of 
$\Gamma$ made of all possible
products of elements in $X$ with elements in $Y$: 
 \begin{equation}\label{eq:star} X \star Y := \{xy, (x, y) \in 
(\Gamma \times_{t,M,s} \Gamma) \cap (X \times Y) \} \end{equation}
The product $\star$ is of course associative:
 $$ X \star (Y \star Z) = (X \star Y) \star Z$$
for all $ X,Y,Z \subset \Gamma$.

In order to express left translation, right translation, and then conjugation on $\Gamma$, one 
is led to the concept of a bisection.

\begin{defi}  \label{bisections_in_a_groupoid}
Let $\xymatrix{\Gamma\ar@<2pt>[r]^{s}\ar@<-2pt>[r]_{t} & M}$ be a groupoid over $M$.
\begin{enumerate}
\item[(a)] A {\it bisection} of $\Gamma$ is any of the two equivalent data:
   \begin{enumerate}
	\item[(i)] a subset $\Sigma \subset \Gamma$ such that (the restrictions of) source and 
target maps $s:\Sigma \to M$ and $t:\Sigma \to M$ are bijective. 
	\item[(ii)] a map $\sigma:M\to \Gamma$ which is right inverse to the source map $s$ 
(i.e. a  section of $s$) and for which $t\circ\sigma:M \to M$ is a bijection of $M$.
	\end{enumerate}
Bisections seen as subsets shall be in general denoted by capital greek letters and their
corresponding maps by the corresponding greek script letters, e.g. $\Sigma$ and $\sigma$, $\Tau$ and $\tau$. 	 
\item[(b)] For a given bisection $\Sigma$, the {\it left translation} induced by $\Sigma$ is
the bijection of $\Gamma$ given by: 
$$l_{\Sigma}:\Gamma \to \Gamma,\,\,\,\,\,\gamma \mapsto \sigma \big((t\circ\sigma)^{-1}(s(\gamma))\big)
\,\,\gamma = \Sigma \star \{\gamma \}.$$ 
\item[(c)] Similarly, the {\it right translation} induced by $\Sigma$ is
$$r_{\Sigma}:\Gamma\to \Gamma,\,\,\,\,\,\gamma\mapsto 
\gamma\,\,\sigma(t(\gamma))=  \{\gamma \} \star \Sigma.$$
\item[(d)] Combining both, the {\it conjugation} induced by $\Sigma$ is
$$c_{\Sigma}:\Gamma\to \Gamma,\,
\gamma\mapsto  \sigma \big((t\circ\sigma)^{-1}(s(\gamma))\big)  
\,\,\gamma\,\,\tilde{\sigma}(t(\gamma)) = \Sigma \star \{ \gamma \} 
\star \Sigma^{-1},$$
where $\tilde{\sigma}(m)=(\sigma ( t\circ\sigma)^{-1}(m))^{-1}$ is 
the section of $s$ associated to the bisection $\Sigma^{-1}$.
\end{enumerate}
\end{defi}
 
For every Lie groupoid $\Gamma$, bisections form a group. The group product of bisections 
$ \Sigma, \Tau$, seen as subsets of $\Gamma$, is simply given by $ \Sigma \star \Tau$.

\begin{rem}
There is another reason to be interested in bisections which is related to the integration of 
Lie algebroids. Indeed, given a section $a\in\Gamma (A)$ of a Lie algebroid $A\to M$, consider 
the local flow $t \to \Phi_t^a$ of the left invariant vector field on  $\Gamma $ associated to 
$a$. As explained in Appendix A in \cite{CraFer},
for $t$ small enough and $a$ compactly supported, the submanifold $\Sigma_t := \Phi_t^a(M)$ is 
a local bisection of the corresponding Lie groupoid. The section of the source map associated 
with it is by construction the restriction of 
$\Phi_t^a $ to $M$. As a consequence, bisections can be seen as being obtained from integrating 
sections of a Lie algebroid.
\end{rem} 

\section{Rackoids}

In a first approach to the definition of Lie rackoids, we will discuss in this section the discrete 
version of a Lie rackoid, called simply a rackoid. The manifold version, i.e. Lie rackoids, will 
come up later. Note that not all rackoids that we are aware of are Lie rackoids, 
see propositions \ref{prop:conjclassesAreRackoids} and corollary 
\ref{rackoide_sans_isotropie} below.

\subsection{Racks}

First recall the notion of a rack, generalizing the conjugation operation in a group.

\begin{defi}
A \emph{rack} consists of a set $X$ equipped with a binary operation denoted 
$(x,y)\mapsto x\rhd y$ such that for all $x, y,$ and $z\in X$, the map $y\mapsto x\rhd y$ is 
bijective and
$$x \rhd (y\rhd z)\,=\,(x\rhd y)\rhd (x\rhd z).$$
\end{defi} 

Obviously, a group $G$ with the conjugation operation $g\rhd h:=ghg^{-1}$ 
for all $g,h\in G$ is an example of a rack. Observe that any conjugacy class
or union of conjugacy classes is still a rack for this operation, 
while it is not a group in general. Let us recall the notion of an
augmented rack:

\begin{defi}
Let $G$ be a group and $X$ be a $G$-set. We say that $X$ together with a 
map $p:X\to G$ is an \emph{augmented rack} when
it satisfies the augmentation identity, i.e. 
$$p(g\cdot x)\,=\,g\,p(x)\, g^{-1}$$
for all $g\in G$ and all $x\in X.$
\end{defi}

An augmented rack is not strictly speaking a rack,
but gives rise to a genuine rack structure on the set $X$:

\begin{prop}    \label{augmented_rack}
Suppose $p:X\to G$ is an augmented rack. Then  
setting for all $x,y\in X$, 
$$x\rhd y\,:=\,p(x)\cdot y$$
endows $X$ with the structure of a rack. 
\end{prop}

\begin{proof}
We give a proof here for further reference. For all $x,y,z\in X$, a direct computation 
yields the rack identity:
\begin{eqnarray*}
(x\rhd y)\rhd(x\rhd z) &=& p(p(x)\cdot y)\cdot(p(x)\cdot z) \\
&=&  (p(x)p(y)p(x)^{-1})\cdot(p(x)\cdot z) \\
&=&  p(x)\cdot(p(y)\cdot z) \\
&=&  x\rhd(y\rhd z).
\end{eqnarray*}
\end{proof}

We will see in section \ref{augmented_rackoids} how this construction mechanism generalizes 
to rackoids. Units can be considered in the framework of racks; the corresponding notion is that 
of a pointed rack: 

\begin{defi}
A rack $X$ is {\it pointed} in case there exists an element $1\in X$ such that for all $x\in X$ 
$$1\rhd x\,=\,x,\,\,\,\,\,\,\,\,x\rhd 1\,=\,1.$$
\end{defi} 

Once again, the conjugation rack of a group is an example of a pointed rack. 

\begin{defi}
A {\it Lie rack} $X$ is a manifold which carries a pointed rack structure
such that all the structure maps are smooth and for all $x\in X$, the smooth 
map $y\mapsto x\rhd y$ is a diffeomorphism.   
\end{defi}

We take Lie racks to be pointed in order to be able to construct their tangent 
Leibniz algebra on the tangent space at the distiguished point 
(see Theorem 3.4 in \cite{Kin}).  

\subsection{Definition of a rackoid}

\begin{defi}
A small {\it precategory} is a pair of sets $(\Gamma,M)$ together with
surjective maps $s,t:\Gamma\to M$ and a map 
$\epsilon:M\to\Gamma$ such that $s\circ\epsilon=t\circ\epsilon=\id_M$.
In general, we shall use the shorthand $ 1_m =\epsilon(m)$.

\end{defi}

Obviously, a precategory with a composition 
map ${\rm comp}:\Gamma\times_M\Gamma\to\Gamma$
is a category. A pair of sets $(\Gamma,M)$ together with
surjective maps $s,t:\Gamma\to M$ will also be called a {\it semi-precategory}.
This notion will be important for the non-unitary version of rackoids.  

Definition \ref{bisections_in_a_groupoid} still makes sense for semi-precategories.
We repeat it for convenience:

\begin{defi}  \label{def:bisections}
Let $\xymatrix{\Gamma\ar@<2pt>[r]^{s}\ar@<-2pt>[r]_{t} & M}$ be a semi-precategory over $M$.
A {\it bisection} of $\Gamma$ is any of two equivalent data:
   \begin{enumerate}
	\item[(i)] a subset $\Sigma \subset \Gamma$ such that (the restrictions of) source and 
target maps $s:\Sigma \to M$ and $t:\Sigma \to M$ are bijective. 
	\item[(ii)] a map $\sigma:M\to \Gamma$ which is right inverse to the source map $s$ and 
for which $t\circ\sigma:M \to M$ is a bijection of $M$.
	\end{enumerate}
\end{defi}

Of course, when a groupoid is considered as a semi-precategory, bisections in the sense
of Definition \ref{def:bisections} are exactly the bisections described in section 
\ref{bisections_in_a_groupoid}.

We now define rackoids. For the sake of clarity, let us fix or recall some notation. For 
every pair of points $m,n \in M$, we denote by $ \Gamma_m^n$ the set of all elements in 
$\Gamma$ with source $m$ and target $n$. Again, for every bisection $\Sigma$,  we denote by 
$\sigma$ the corresponding right inverse of $s:\Sigma \to M$ and by $\underline{\sigma}: M \to M$
the bijection of $M$ obtained as the composition $ t \circ \sigma $.

\begin{defi}  \label{definition_rackoid}
A \emph{(non-unital) rackoid} is a semi-precategory $\xymatrix{\Gamma\ar@<2pt>[r]^{s}
\ar@<-2pt>[r]_{t} & M}$ with a composition law 
$\rhd:(\Sigma,\gamma)\mapsto \Sigma\rhd\gamma$ mapping
a bisection $\Sigma$ and an element $\gamma \in \Gamma_m^n$
to an element in $\Gamma_{\underline{\sigma} (m)}^{\underline{\sigma} (n)} $.
\footnote{Equivalently, for all bisections $\Sigma$ and all elements $\gamma \in \Gamma$, 
the composition $ \sigma \rhd \gamma$ is defined to be an element of $\Gamma$ whose source 
is $\underline{\sigma} \circ s (\gamma) $ and whose target is $\underline{\sigma} \circ t (\gamma) $.}
We require that:
\begin{enumerate}
\item for all bisections $\Sigma$, the assignment $\Sigma\rhd-:\Gamma \to \Gamma $
is a bijection\footnote{It is then automatic that for every bisection $ \Tau$, the image of 
$\Tau \subset \Gamma$ under
 $\Sigma\rhd- $, subset that we shall denote by $\Sigma\rhd \Tau $, is a bisection again, and 
that $\underline{\sigma\rhd \tau} = \underline{\sigma} \circ \underline{\tau} \circ \underline{\sigma}^{-1}$.}, 
\item the composition law is supposed to satisfy the self-distributivity relation
\begin{equation}   \label{eq:self_distributivity}
\Sigma\rhd(\Tau\rhd \gamma)=(\Sigma\rhd\Tau)\rhd(\Sigma\rhd \gamma)
\end{equation}
for all bisections $\Sigma,\Tau$ and all $\gamma\in\Gamma$. 
\end{enumerate}

Furthermore, in order to define a {\it unital or pointed rackoid},
in case  $ \Gamma$ is a precategory, we require the composition law to satisfy 
$1_M\rhd\gamma=\gamma$ and $\sigma\rhd 1_m=1_{\underline{\sigma}(m)}$ for all $m\in U$. Here we write 
$1_m$ for the unit at $m$, i.e. $1_m=\epsilon(m)$ and by $1_M$ the bisection $\epsilon (M)$. 
\end{defi}

As expected, a rackoid over a point (i.e. when $M$ is a point in definition \ref{definition_rackoid}) 
is a rack. It is not true in general that the vertex- or isotropy sets 
$\Gamma^m_m:=s^{-1}(m)\cap\,t^{-1}(m)$ become racks.

This is true, however, under the following conditions:

\begin{prop}
Let $\Gamma$ be a rackoid over $M$ and $m\in M$ be some element such that $\Gamma_m^m\not=\emptyset$.
Assume that there is a bisection through each point $ \gamma' \in \Gamma_m^m $ 
and that $\Sigma_1 \rhd \gamma = \Sigma_2 \rhd \gamma$ for every bisections 
$\Sigma_1,\Sigma_2 $ through $\gamma'$, then the isotropy sets
$\Gamma^m_m:=s^{-1}(m)\cap\,t^{-1}(m)$ become racks via the induced operation.
\end{prop} 

The rack $\Gamma^m_m$ described in the previous proposition is called {\it isotropy rack} 
in at $m$ of the rackoid $\Gamma$. 

Rackoids have permanence properties quite different from those of groupoids or even of racks:

\begin{prop}
For all rackoid $\Gamma$, the subset $ \{ \gamma \in \Gamma | s(\gamma) \neq t(\gamma) \} $
is a rackoid. 
\end{prop}
\begin{proof}
If $x\in\Gamma$ has $s(x)=m\in M$ and $t(x)=n\not= m$ with $n,m\in U$, then 
for all bisections $\sigma\in{\rm Bis}(\Gamma)$, we have that
$s(\sigma\rhd x)=\underline{\sigma}(m)$ and 
$t(\sigma\rhd x)=\underline{\sigma}(n)\not=\underline{\sigma}(m)$
by bijectivity.
Therefore the subset of elements where source and target are different 
is preserved by the rack operations.  
\end{proof} 

We will not admit the empty set as a rackoid and we will furthermore implicitely assume that all our
(semi-) precategories do admit bisections (otherwise the axioms are empty).

\subsection{Groupoids as rackoids}

The first example of a rackoid is of course a groupoid. For $\Gamma$ a groupoid over $M$, the 
composition rule defined for all bisection $\Sigma$ and $\gamma$, by 
\begin{equation}\label{eq:product-bisection-elements} 
\Sigma \rhd \gamma = \Sigma \star \{ \gamma \} \star \Sigma^{-1}
\end{equation}
defines a rackoid structure. Recall that $ \star $ is the operation defined in (\ref{eq:star}),
see Section \ref{bisections_in_a_groupoid}. 

\begin{prop}  \label{groupoid_as_rackoid}
The conjugation operation (\ref{eq:product-bisection-elements}) in a groupoid gives rise to a rack 
product on bisections, rendering the groupoid a rackoid.
\end{prop}
\begin{proof}
This follows directly from Section \ref{bisections_in_a_groupoid} and definition 
\ref{definition_rackoid} which is modeled onto the conjugacy rackoid underlying a groupoid. 
\end{proof} 

By the \emph{orbit} of an element $\gamma$ in a groupoid $\Gamma$, we mean an orbit of $s(\gamma)$
under the natural action
of a groupoid on its unit space. Isotropy groups over elements in a given orbit $O$ form 
a group bundle over $M$. For all $m,n \in M$, conjugation by an arbitrary element $\gamma \in \Gamma_m^n $
induces a group automorphism and different choices for $\gamma$ lead to automorphisms that differ by 
inner group
morphisms.  In particular, the isotropy groups associated to any two elements  
in the same orbit are conjugated.

By the \emph{conjugacy class} of an element $\gamma$ in a groupoid $ \Gamma$ that we assume to 
admit a bisection through any of its element, we mean the set of all 
elements of the form $\Sigma \rhd \gamma $ for $\Sigma$ a bisection of $\Gamma$. If 
$ s(\gamma) \not= t(\gamma) $, then the conjugacy class of $\gamma$ is made of all elements 
$ \{ \gamma' \in \Gamma_O^O \, | \, s(\gamma') \not= t(\gamma') \} $ with $O$ being the orbit of 
$\gamma$. If $ s(\gamma) = t(\gamma) $, then the conjugacy class of $\gamma$ is made of the 
collection of all the orbits in the isotropy groups of the orbits of $\gamma$ that 
correspond to the conjugacy group of $ \gamma \in \Gamma_m^m$. In any case, the following 
result is true:

\begin{prop}\label{prop:conjclassesAreRackoids}
Each conjugacy class in a groupoid is a rackoid.
\end{prop}

\begin{cor}  \label{rackoide_sans_isotropie}
Identify, in a Lie groupoid, two elements $\gamma$ and $\gamma'$ if and only if 
$s(\gamma)=s(\gamma')$ and $t(\gamma)=t(\gamma')$. The set
of equivalence classes is a rackoid (and even a groupoid).
\end{cor}

\subsection{Augmented rackoids}   \label{augmented_rackoids}

In the present section, we generalize to rackoids the mechanism which constructs racks out of 
augmented racks. 

\begin{theo}\label{thm:augmented_rackoids}
Let $\xymatrix{\Gamma\ar@<2pt>[r]^{s'}\ar@<-2pt>[r]_{t'} & M}$ be a groupoid and 
$\xymatrix{X\ar@<2pt>[r]^{s}\ar@<-2pt>[r]_{t} & M}$ be 
a precategory. Suppose that there exists a morphism of precategories $p:X\to \Gamma$, 
i.e. a map such that
\begin{enumerate}
\item The map $p$ intertwines\footnote{Note that this condition implies that the $p$-image 
of a bisection of $X$ is a bisection of $\Gamma$.} the sources and targets of 
$X$ and $ \Gamma$:
$$\xymatrix{X \ar@<2pt>[d]^{s}\ar@<-2pt>[d]_{t} \ar[r]^p & 
\Gamma \ar@<2pt>[d]^{s'} \ar@<-2pt>[d]_{t'}  \\
M \ar[r]^{\rm id_M} & M }$$
\item The map $p$ intertwines the identity maps $\epsilon:M\to X$ and $\epsilon':M\to\Gamma$:
$$\xymatrix{ X \ar[r]^p & \Gamma \\
             M \ar[u]^{\epsilon} \ar[r]^{\rm id_M} & M \ar[u]^{\epsilon'} }$$ 
\end{enumerate}
Suppose that the group of bisections $\Bis(\Gamma)$ acts on $X$ in such a manner that for all 
$\Sigma \in \Bis (\Gamma)$ and all $x\in X$: 
\begin{equation}    \label{augmentation_identity}
p(\Sigma \cdot x)\,=\,\Sigma \star p(x) \star \Sigma^{-1},
\end{equation}
while for all $m \in M$
\begin{equation}    \label{augmentation_identity_2}
p(\Sigma) \cdot 1_m'=1_m'. 
\end{equation}
where, as usual, $1_m'=\epsilon'(m)$.

Then the prescription $$ \Tau\rhd x\,:=\,p(\Tau)\cdot x $$ defines for all $x\in X$ and all 
bisections $\Tau$ of $X$ a pointed rackoid structure on $X$. 
\end{theo}

\begin{proof}
Observe that the first commutative diagram implies that $p$ induces a map
$$p:{\rm Bis}(X)\to {\rm Bis}(\Gamma)$$
between sets of  bisections. Observe further that the augmentation identity
(\ref{augmentation_identity}) implies that
$s(g\cdot x)=\underline{g}(s(x))$ and $t(g\cdot x)=\underline{g}(t(x))$ where 
$\underline{g}$ is the diffeomorphism associated to the  bisection $g$. 
The autodistributivity relation follows as in the proof of Proposition 
\ref{augmented_rack} from the augmentation identity
(\ref{augmentation_identity}). The map $p$ sends identities to identities, i.e. 
(\ref{augmentation_identity_2}) holds, thus 
for all $m\in M$ and all $x,y\in X$:
$$ 1_m\rhd y\,=\,p(1_m)\cdot y\,=\,1_m'\cdot y\,=\,y, $$
and by hypothesis,
$$ x\rhd 1_m\,=\,p(x)\cdot 1_m'\,=\,1_m'. $$
\end{proof} 

A context where this can be applied is the following; we use here the non-unitary version of 
a rackoid and the corresponding non-unitary version of the above theorem, formulated in terms of 
semi-precategories instead of precategories.  

Let $q:Y\to M$ be a surjective map.  
Let $ \Gamma$ be a groupoid $\xymatrix{\Gamma\ar@<2pt>[r]^{s'}\ar@<-2pt>[r]_{t'} & M}$ that acts 
freely and transitively on the fibers of $q$. Consider the fiber product $X:=Y\times_{q,M,q}Y$ and define 
$$ p:X\to \Gamma,\,\,\,\,\,x=(y_1,y_2)\mapsto \gamma, $$
by choosing as $\gamma$ the unique element such that $y_1=\gamma\cdot y_2$. Let a bisection 
$\sigma\in{\rm Bis}(\Gamma)$, seen as a section of the source map $s$, act on $X$ by
$$\sigma\cdot(y_1,y_2)\,:=\,(\sigma(y_1)\cdot y_1,\sigma(y_2)\cdot y_2),$$
where one multiplies with the unique element $\sigma(y)$ in the bisection $\sigma$ such that 
the composition makes sense. By construction,  
$$p(\sigma\cdot(y_1,y_2))\,:=\,\sigma(y_2)p(y_1,y_2)\sigma(y_1)^{-1}.$$
Applying the (non-unitary version of the) above theorem yields a (non-unitary) rackoid structure 
on $X=Y\times_{q,M,q}Y$. We will exploit this example in the smooth framework in section \ref{sec:link}.  

\section{Lie rackoids}  

\subsection{Manifold structure on bisections}

In the following, we will work with smooth precategories (and - without explicitely transcribing 
everything - with smooth semi-precategories). 

\begin{defi}
A {\it smooth precategory} is a pair of smooth manifolds $(\Gamma,M)$ together with surjective 
submersions $s,t:\Gamma\to M$ and a smooth map $\epsilon:M\to\Gamma $ (mapping $ m \in M$ to 
$ 1_m =\epsilon(m) \in \Gamma $) such that 
$s\circ\epsilon=t\circ\epsilon=\id_M$.
\end{defi}

Here for "smooth manifold" we break with the tradition which admits Lie groupoids such that 
$\Gamma$ is a non-necessarily Hausdorff, non-necessarily second countable smooth manifold.
In our setting, "smooth manifold" always means that $\Gamma$
is a Hausdorff, non-necessarily second countable (because $\Gamma$ can be an 
infinite-dimensional manifold) smooth manifold, while
$M$ is supposed to be a Hausdorff, second countable smooth manifold. In fact, we will 
always suppose $M$ to be compact. Here comes the definition of bisections in the 
smooth framework. 

\begin{defi} \label{definition_bisection}
A {\it smooth bisection} of a smooth precategory 
$\xymatrix{\Gamma\ar@<2pt>[r]^{s}\ar@<-2pt>[r]_{t} & M}$ is
 a  bisection $\Sigma\subset\Gamma$ such that the associated right inverse to $s:\Gamma\to M$
is a smooth map and such that the bijective map $\underline{\sigma}$ is a diffeomorphism. 
\end{defi}

Notice that smooth bisections are precisely the submanifolds of $\Gamma$ to which the 
restriction of both $s$ are $t$ are diffeomorphism onto $M$.

In the special case where $\xymatrix{\Gamma\ar@<2pt>[r]^{s}
\ar@<-2pt>[r]_{t} & M}$ is a Lie groupoid, the smooth bisections from the above 
definition are exactly the smooth bisections we have discussed earlier. 
  
Let us still denote by ${\rm Bis}(\Gamma)$ the set of all  
smooth bisections.
We also still denote by $\underline{\sigma}=t\circ\sigma : m\to t(\sigma(m))$ 
the corresponding diffeomorphism of $M$.

Let us show that the set of (smooth) bisections of a smooth precategory (and also of a smooth semi-precategory)
has the structure of an infinite-dimensional manifold. This structure is closely related to the Lie 
group structure on the set of bisection of a Lie groupoid by Schmeding and Wockel, see \cite{SchWoc}.  

\begin{prop}\label{prop:Frechet}
Let $\xymatrix{\Gamma\ar@<2pt>[r]^{s}\ar@<-2pt>[r]_{t} & M}$ be a smooth precategory with compact 
base manifold $M$. Then the set of bisections $Bis(\Gamma)$ carries a structure of a Fr\'echet manifold. 
\end{prop} 

\begin{proof}
By Proposition 10.10 in \cite{Mic}, the space of all sections $S_s(M,\Gamma)$ of the surjective submersion 
$s:\Gamma\to M$ is a splitting submanifold of the Fr\'echet manifold ${\mathcal C}^{\infty}(M,\Gamma)$
(equipped with the ${\mathcal C}^{\infty}$ topology). 

On the other hand, the composition with the smooth target map $t$ is a smooth map 
$$t_*:{\mathcal C}^{\infty}(M,\Gamma)\to{\mathcal C}^{\infty}(M,M),$$
and this remains true for its restriction to the submanifold $S_s(M,\Gamma)$. By Corollary 5.7
of \cite{Mic}, the subgroup of diffeomorphisms ${\rm Diff}(M)\subset{\mathcal C}^{\infty}(M,M)$
is open and acquires thus its Fr\'echet manifold structure. By construction, the set 
$Bis(\Gamma)=(t_*)^{-1}({\rm Diff}(M))$ is therefore an open submanifold of $S_s(M,\Gamma)$, and thus 
an open submanifold of ${\mathcal C}^{\infty}(M,\Gamma)$.
\end{proof}   

\begin{cor}  \label{families_of_bisections}
With respect to this manifold structure on $Bis(\Gamma)$, a family of bisections $\sigma_{s}$ with $s$ in a 
neighborhood $ {\mathcal U}$ of $0$ in ${\mathbb R}^n $  is smooth if and only if the function 
$ (u,m) \to \sigma_u (m) $ is a smooth function from ${\mathcal U} \times M $ to $ \Gamma$.
\end{cor}

\begin{proof}
This follows immediately from the so-called exponential law, see Theorem A in \cite{Glo}.
\end{proof} 

\subsection{Definition of Lie rackoids}

Here comes the definition of a Lie rackoid, i.e. the smooth version of a rackoid.
We suppose in the following definition the set of bisections ${\rm Bis}(\Gamma)$
to be endowed with the Fr\'echet manifold structure defined in the preceding subsection.  
 
\begin{defi}
A Lie rackoid is a smooth precategory $\xymatrix{\Gamma\ar@<2pt>[r]^{s}
\ar@<-2pt>[r]_{t} & M}$ with a smooth composition law 
$\rhd:(\Sigma,\gamma)\mapsto\Sigma\rhd\gamma\in\Gamma$ for all
bisections $\Sigma\in{\rm Bis}(\Gamma)$ and all $\gamma\in\Gamma$ such that 
$\Sigma\rhd-:\Gamma\to\Gamma$
is a smooth diffeomorphism for all $\Sigma\in{\rm Bis}(\Gamma)$. 

This composition law
is supposed to satisfy the self-distributivity relation
\begin{equation}   \label{*}
\Sigma\rhd(\Tau\rhd \gamma)=(\Sigma\rhd\Tau)\rhd(\Sigma\rhd \gamma)
\end{equation}
for all $\Sigma\in{\rm Bis}(\Gamma)$, $\Tau\in{\rm Bis}(\Gamma)$ and all 
$\gamma\in\Gamma$. 

The composition law is supposed to be compatible with the source and target map
in the sense that 
$$s(\Sigma\rhd\gamma)=\underline{\sigma}(s(\gamma)),\,\,\,\,\,\,
t(\Sigma\rhd\gamma)=\underline{\sigma}(t(\gamma))$$
for all $\Sigma\in{\rm Bis}(\Gamma)$ and all $\gamma\in\Gamma$.

Furthermore, the composition law should
satisfy $1_M\rhd\gamma=\gamma$ with $1_M=\epsilon(M)$ and 
$\Sigma\rhd 1_m=1_{\underline{\sigma}(m)}$ for all $m\in M$.   
\end{defi}

Observe that a Lie rackoid is not a rackoid in the sense of Definition 
\ref{definition_rackoid} as the rack product is only defined on the set of smooth
bisections and not on the set of all bisections.  

\begin{rem}   \label{remark_global}
We have defined a Lie rackoid only using global bisections, because we wanted to 
avoid that the rackoid-product of a bisection $\Sigma$ and a $\gamma\in\Gamma$ depends
only on the germs of the bisection at the source and the target of $\gamma$. 
In our examples, it 
turns out to depend on the whole bisection. Nevertheless, the corresponding local 
definition (i.e. using only local bisections) for
a Lie rackoid also makes sense and gives a (a priori) different version of Lie rackoid
where the product does only depend on the germs of the bisections at source and target.  
\end{rem}

As before, we have:

\begin{prop}
Let $\Gamma$ be a Lie rackoid over $M$ and $m\in M$ be some element such that $\Gamma_m^m\not=\emptyset$.
Assume that there is a smooth bisection through each point $ \gamma' \in \Gamma_m^m $ 
and that $\Sigma_1 \rhd \gamma = \Sigma_2 \rhd \gamma$ for every bisections 
$\Sigma_1,\Sigma_2 $ through $\gamma'$, then the isotropy sets
$\Gamma^m_m$ become Lie racks via the induced operation. 
\end{prop}

Note that the hypothesis of possessing enough smooth bisections 
(i.e. at least one through every point of the groupoid) is for example fulfilled
in the case of a source-connected Lie groupoid, see \cite{CLZ}.

\section{Bundles of Lie racks as Lie rackoids}

Let us here discuss a very explicit example of a Lie rackoid,
namely a bundle of Lie racks. In this case we have $s=t$
for the source and target maps of the underlying precategory. 

The following proposition is straighforward.
\begin{prop}
A bundle of Lie racks endows a natural Lie rackoid structure, 
with respect to the product given by 
$\Sigma \rhd \gamma  =  \sigma (m) \rhd_m \gamma$ for all $m \in M$, with $\gamma$
an arbitrary point in the fiber over $m$, fiber whose rack product we denote by 
$\rhd_m $, and $\Sigma $ a bisection, seen as a section $\sigma$ of the 
projection onto the base manifold.
\end{prop}

Below is an example of non-trivial bundle of Lie rack.

The idea is here to construct a ``jump deformation'' from a 
$4$-dimensional very elementary example of a Lie rack 
(due to S. Covez in his PhD thesis (p.78)). 
Namely, for $t\in\R$, let 
${\mathfrak g}_t:=\R^4$ 
be the Leibniz algebra defined by the bilinear map
$$[,]_t\,:\,{\mathfrak g}_t\times {\mathfrak g}_t\to{\mathfrak g}_t$$
which is set to be 
$$[(x_1,x_2,x_3,x_4),(y_1,y_2,y_3,y_4)]_t:=(0,0,0,tx_1y_1+x_1y_2-x_2y_1+
tx_2y_2+tx_3y_3).$$
Observe that the terms which render this Leibniz bracket non-Lie are 
multiplied by $t\in\R$, so for $t=0$, this is a genuine Lie algebra, while 
for $t\not=0$, it is a non-Lie Leibniz algebra. As Covez explains, 
${\mathfrak g}_1$ (and thus also ${\mathfrak g}_t$ for $t\not=0$) is 
a non-split Leibniz algebra, i.e. it is not isomorphic to a hemisemidirect product of 
a Lie algebra and a representation. 

Covez integrates the Leibniz bracket $[,]_t$ on $\R^4$ into a Lie rack 
structure on $\R^4$. Our idea is to do this here ``in families'', the 
non-triviality of our Leibniz algebra bundle coming from the fact that 
it degenerates to a Lie algebra at $t=0$. Obviously, the two parts 
$({\mathfrak g}_t)_{\rm Lie}$
and $Z_L({\mathfrak g}_t)$ integrate to the trivial Lie racks 
$\R^3$ and $\R$ respectively. 

Writing down the explicit formula for the rack cocycle $f_t$ obtained 
from integrating the Leibniz 
cocycle $\omega_t$, it turns out that for all 
$a=(a_1,a_2,a_3,a_4), b=(b_1,b_2,b_3,b_4)\in\R^4$, $f_t(a,b)=\omega_t(a,b)$,
the explicit formula for the Lie rack structure $\rhd_t$ on $\R^4$ is
$$(a_1,a_2,a_3,a_4)\rhd_t(b_1,b_2,b_3,b_4)=(b_1,b_2,b_3,ta_1b_1+ta_2b_2+
ta_3b_3+a_1b_2-a_2b_1+b_4).$$

Observe that for $t=0$ the rack structure is the conjugation in the 
standard semi-direct product group
$$(a_1,a_2,a_3,a_4)\cdot(b_1,b_2,b_3,b_4)\,=\,(a_1+b_1,a_2+b_2,a_3+b_3,a_4+b_4+a_1b_2).$$    

\begin{theo}
The manifold $\R^5$ with the above structure $\rhd_t$ and the projection onto $\R$ given
by $(a_1,a_2,a_3,a_4,t)\mapsto t$ is a (non-trivial) bundle of Lie racks, i.e. it is not 
isomorphic (as Lie racks) to a direct product of the trivial Lie rack $\R$ 
with some Lie rack $\R^4$.    
\end{theo}

\pr The proof is very simple: In case the bundle was trivial, the isomorphism type of the 
Lie rack should be constant. But this is not the case, as the Lie rack degenerates to a 
conjugation rack coming from a Lie group at $t=0$.\fin 

\section{Lie rackoids generalize Lie groupoids} 

In this section, we highlight two propositions which show in which sense Lie 
rackoids generalize Lie groupoids: Any Lie groupoid gives rise to a Lie
rackoid via conjugation, and any Lie rackoid defines a tangent Leibniz
algebroid. 

\subsection{The Lie rackoid underlying a Lie groupoid}  

In the smooth framework, we have the following proposition which strenghtens 
Proposition \ref{groupoid_as_rackoid}. 

\begin{prop}
A Lie groupoid defines a Lie rackoid via conjugation.
\end{prop}

\pr A Lie groupoid $\xymatrix{\Gamma\ar@<2pt>[r]^{s}
\ar@<-2pt>[r]_{t} & M}$ has obviously an underlying smooth precategory
$\xymatrix{\Gamma\ar@<2pt>[r]^{s}\ar@<-2pt>[r]_{t} & M}$. The bisections
(with respect to the precategory structure) are then just the ordinary 
bisections of the Lie groupoid (see Definition 3.2 and
\cite{Mack} p. 22).
Observe that for each bisection $\Sigma\in{\rm Bis}(\Gamma)$, we have an inverse 
bisection $\Sigma^{-1}\in{\rm Bis}(\Gamma)$ (see {\it loc. cit.} p. 22), 
still right inverse to the source map,
but such that $t\circ\sigma^{-1}$ is the inverse of the diffeomorphism 
$t\circ\sigma:M\to M$. 

We define for all
$\Sigma\in{\rm Bis}(\Gamma)$ and all $\gamma\in\Gamma$ 
the composition law $\rhd:(\Sigma,\gamma)\mapsto\Sigma\rhd\gamma\in\Gamma$ 
using the conjugation, i.e. 
$\Sigma\rhd\gamma:=\Sigma\gamma\Sigma^{-1}$. It is clear (from the 
associativity of the groupoid operation for composable elements) that $\rhd$
satisfies the self-distributivity relation. It is still more obvious
that $1_M\rhd\gamma=\gamma$ with $1_M  =\epsilon(M)$ and 
$\Sigma\rhd 1_m=1_{\underline{\sigma}(m)}$ for all $m\in M$. 

The operation $\rhd$ is clearly smooth with respect to the Fr\'echet manifold structure 
on ${\rm Bis}(\Gamma)$ by Corollary \ref{families_of_bisections}. \fin

\subsection{From Lie rackoids to Leibniz algebroids}

Recall the definition of a Leibniz algebroid from \cite{ILMP}.

\begin{defi}
A Leibniz algebroid is the data of a vector bundle $\pi:A\to M$ together with a Leibniz
algebra structure on the space of global sections $\Gamma(A)$ and a bundle 
morphism $\rho:A\to TM$
(the anchor) such that $\rho$ (on the level of sections) is a Leibniz morphism and
for all $b,a\in\Gamma(A)$ and all $f\in{\mathcal C}^{\infty}(M)$
\begin{equation}\label{eq:anchor} [b,fa]\,=\,f [b,a]+\rho(b)(f) \, a.\end{equation}
\end{defi}  

Note that, since the bracket is a priori not skew-symmetric, there is no way to express 
$[fs_1,s_2]  $, and it is has no reason to be $\, f[s_1,s_2] \, -\, \rho(s_2)(f) \, s_1$ 
in general. 

From now, and until the end of this section, we assume that $M$ is a compact manifold.

\begin{theo}\label{thm:infinitesimal}
A Lie rackoid over a compact manifold gives rise to a tangent Leibniz algebroid.
\end{theo}

We shall  give a more precise statement in Theorem \ref{thm:infinitesimal2} below.

Let $\xymatrix{\Gamma\ar@<2pt>[r]^{s} \ar@<-2pt>[r]_{t} & M}$ be a pointed Lie rackoid. 
Consider  the pull-back $\epsilon^* T \Gamma$ of $T\Gamma \to \Gamma$ through 
$\epsilon: M \hookrightarrow \Gamma$.
By construction, the fiber of $\epsilon^* T \Gamma$ over $m \in M$ is $T_{1_m} M$.
There are therefore two submersive vector bundle maps $Ts,Tt$ from 
$\epsilon^* T \Gamma$ to $TM$, obtained by differentiating 
the source and target maps
 $$ T_{1_m}t : T_{1_m} \Gamma \to T_{m} M \hbox{ and } T_{1_m}s : T_{1_m} \Gamma \to T_{m} M  .$$
We define a vector bundle $A$ that we call \emph{infinitesimal Leibniz algebroid} by considering
 the kernel of $Ts : \epsilon^* T \Gamma \to TM $.
 In equation:
 \begin{equation}\label{eq:fibre}   A:={\rm ker}(Ts) = \coprod_{m\in M}\ker(T_{1_m}s)\subset\coprod_{m\in M}T_{1_m}\Gamma=\epsilon^* T\Gamma.\end{equation}
We then define the \emph{anchor map} to be a vector bundle morphism $\rho$ from $A$ to $TM$ obtained by restricting to 
$A \subset \epsilon^* T\Gamma $ the vector bundle submersion $Tt: \epsilon^* T \Gamma \to TM$, up to a sign.
In equation:
 \begin{equation}\label{eq:ancre}   \rho = - \left. Tt  \right|_{A} .\end{equation}
So far, the construction is parallel to the construction of the Lie algebroid associated to a Lie groupoid.

We shall need the next technical lemma, in which $I$ is a shorthand for the  interval  $]-1,+1[$,
the topology on $\Bis(\Gamma)$, the Fr\'echet manifold topology used in Proposition \ref{prop:Frechet}
and Corollary \ref{families_of_bisections}, is implicitly used to justify the existence of the 
derivative that appears in it:

\begin{lem}\label{lem:existsbisections}
For every section of $A$ seen as a map from $M$ to $T\Gamma$ mapping $m $ to 
$b_m \in A_m \subset T_{1_m} M $, there exists 
a smooth family $(\Sigma_u)_{u \in I} $ of bisections of $\Gamma$ such that
$ \Sigma_0=\epsilon(M)$ and $\left.\frac{\partial}{\partial u} \right|_{u =0} \sigma_u(m) \subset T_{1_m} \Gamma $
coincides with $ b_m$ for all $m \in M$.
Conversely, for any smooth family $(\Sigma_u)_{u \in I}$ in $\Bis(\Gamma)$,  
such that $ \Sigma_0=\epsilon(M)$,
the assignment $m \mapsto \left.\frac{\partial}{\partial u} \right|_{u =0} \sigma_u(m)  $
is a smooth section of $A$.
\end{lem}
\begin{proof}
Let $\tilde{b}$ be a vector field, supported in a neighborhood of $\epsilon(M) \subset \Gamma$, 
tangent to the fibers of $s$,
and extending $b$, i.e. such that the restriction of $\tilde{b}$ to $\epsilon(M)$ coincides with $ b$.
Let $\sigma_u(m) $ be the flow at the time $u$ starting at $1_m$ of $\tilde{b}$.
It is by construction true that 
$ \left.\frac{\partial}{\partial u} \right|_{u =0} \sigma_u(m) \subset T_{1_m} \Gamma $
coincides with $ b_m$.
Moreover, upon replacing $\tilde{b}$ by $\chi\tilde{b}$ with $\chi$ a function supported in 
a neighborhood of $\epsilon(M) \subset \Gamma$
and equal to $1$ on that manifold, we can assume that $ \sigma_u(m)$ is defined for all 
$u \in I$ and that, for all fixed $u\in I$,
$m \to \sigma_u(m) $  is a bisection $\Sigma_u$ of $\Gamma$. According to Corollary 
\ref{families_of_bisections}, it is a smooth family of bisections, which completes 
the first part of the proof.

The second part follows from Corollary \ref{families_of_bisections}, which grants that  
$$m \mapsto \left.\frac{\partial}{\partial u} \right|_{u =0} \sigma_u(m)  $$ 
is a smooth section of $\epsilon^* T\Gamma$ and the relation $ s \circ \sigma_u = id_M$, which grants
that $\left.\frac{\partial}{\partial u} \right|_{u =0} \sigma_u(m)$ has values in the kernel of 
$T_{1_m}s$, i.e. $A_m$. 
\end{proof}

Lemma \ref{lem:existsbisections} means that, when $\Bis(\Gamma)$ is equipped with the 
Fr\'echet manifold topology as in Proposition \ref{prop:Frechet}:

\begin{prop}\label{prop:existsbisections}
The tangent space of $\Bis (\Gamma)$ at the bisection $\epsilon(M)$ is $\Gamma(A)$.
\end{prop}

We can now define the adjoint action of a bisection $\Sigma \in {\Bis}(\Gamma)$ on $A$.
Let $\underline{\sigma}$ be the diffeomorphism of $M$ associated to $\Sigma$ and let
$$\begin{array}{rrcl} \Sigma^\rhd : &\Gamma &\to & \Gamma \\ 
&  \gamma & \to & \Sigma \rhd \gamma .\end{array} $$
By the definition \ref{definition_rackoid} of a Lie rackoid, the following diagrams commute:
$$ \xymatrix{ M \ar[r]^{ \underline{\sigma}} \ar[d]^\epsilon & M \ar[d]^\epsilon \\ 
\ar[r]^{\Sigma^\rhd} \Gamma & \Gamma }  \hbox{ and } 
\xymatrix{ \Gamma \ar[r]^{\Sigma^\rhd} \ar[d]^{s} & \Gamma  \ar[d]^{s} \\
M  \ar[r]^{  \underline{\sigma}} & M }  .$$
Differentiating at the point $\gamma= 1_m$ (that in the next lines we write 
simply as $m$, making no further reference to the injection $\epsilon$), 
the first of these diagrams, gives that $T_m \Sigma^\rhd$ maps $T_m M$ to $ T_{\underline{\sigma} (m) }M$.
Differentiating at the point $m \in M$ the second of these diagrams yields:
 $$ \xymatrix{ T_{m} \Gamma \ar[r]^{ T_m \Sigma^\rhd }  \ar[d]^{T_m s} & T_{  \underline{\sigma}(m)} \Gamma  
  \ar[d]^{T_{  \underline{\sigma}(m)} s}\\T_m M  \ar[r]^{  T_m \underline{\sigma}} & T_{  \underline{\sigma}(m)} M } $$
which implies that $  T_m \Sigma^\rhd  $ maps the kernel of  $T_m s$ to the kernel of 
$ T_{  \underline{\sigma}(m)} s $,
i.e. maps $A_m$ to $A_{ \underline{\sigma}(m)} $. We call \emph{adjoint map} and 
denote by $\Ad_\Sigma$ the restricted vector bundle morphism.
In equation:
\begin{equation}  \label{def_Ad}
\Ad_\Sigma :=  \left. T_m \Sigma^\rhd \right|_{A}.
\end{equation}
 The inverse $ (\Sigma^\rhd)^{-1} $ of $\Sigma^\rhd$ exists by assumption and its differential 
at $\epsilon (M)$ is the inverse of $\Ad_\Sigma$. Hence,  $\Ad_\Sigma$ is an invertible vector 
bundle morphism over the diffeomorphism $\underline{\sigma}$. In particular, $\Ad_\Sigma$ can 
be seen as a linear operator on the space $\Gamma(A)$ of sections of $A$ that we still denote 
by $\Ad_\Sigma$ with a slight abuse of notation: For every section $a \in \Gamma(A)$, we set 
$ \Ad_\Sigma a$ to be the section of $A$ whose value at $m \in M$ is 
$ \Ad_\Sigma a_{\underline{\sigma}^{-1}(m)} $. In equation:
 \begin{equation}\label{eq:AdjointForSections}  (\Ad_{\Sigma }  \, a )(m)=   
\Ad_\Sigma a_{\underline{\sigma}^{-1}(m)}  \hbox{ for all $m \in M$}.
\end{equation} 
Without going any further, notice that, for every smooth function $f $ on $M$ and every 
section $a $ of $A$:
 \begin{equation}\label{eq:donnera_ancre}  \Ad_{\Sigma }  \, 
fa = (\underline{\sigma}^{-1})^* f \, ( \Ad_{\Sigma} \, a ).
\end{equation}

The adjoint action can also be interpreted as follows.
\begin{lem} \label{lem:adjointAsDifferential}
For every bisection $\Sigma \in \Bis(\Gamma)$, the assignment $\Tau \mapsto \Sigma \rhd \Tau$
is a smooth diffeomorphism of $\Bis(\Gamma)$, mapping $\epsilon(M)$ to itself, and whose 
differential at the point 
$ \epsilon(M)$ is, upon  identifying $T_{\epsilon}(M) \Bis(\Gamma)$ with $\Gamma(A)$ as in 
Proposition \ref{prop:existsbisections},
is the adjoint action $\Ad_\Sigma : \Gamma(A) \to \Gamma(A)$ given by (\ref{def_Ad}).
\end{lem}
This lemma implies immediately the following result:
\begin{lem}\label{lem:smooth}
 The adjoint action  $Bis(\Gamma) \times \Gamma(A) \to \Gamma(A)$ is smooth, 
with $Bis(\Gamma),\Gamma(A)$ being equipped  with the Fr\'echet topology.
 \end{lem}
Notice also that for every pair of bisections $ \Sigma,\Tau$, relation 
(\ref{eq:self_distributivity}) reads
 $$ \Sigma^\rhd \circ \Tau^\rhd = (\Sigma \rhd \Tau)^\rhd \circ \Sigma^\rhd .$$
This implies, when differentiated at a point $m \in M$ in the direction of 
$a \in A_m \subset T_m \Gamma $ that
 \begin{equation}\label{eq:donnera_Leibniz}  \Ad_{\Sigma } \circ  \Ad_{\Tau} \, 
a = \Ad_{ \Sigma \rhd \Tau} \circ \Ad_{\Sigma} \, a .
\end{equation}
Lemma \ref{lem:smooth} allows to consider the differential of the assignment 
$(\Sigma,a) \mapsto \Ad_\Sigma a$ from $\Bis(\Gamma) \times \Gamma(A) \to \Gamma(A)$ 
at the bisection $\Sigma =\epsilon(M)$. Since by Proposition \ref{prop:existsbisections}, 
the tangent space of $\Bis(\Gamma)$ at $ \epsilon(M)$ is precisely $\Gamma(A) $, this 
differential is a continuous assignment $ \Gamma(A) \times \Gamma(A) \to \Gamma(A)$
(with $\Gamma(A)$ equipped with the Fr\'echet topology) which is linear in both variables 
by construction. We call it the \emph{Leibniz algebroid bracket} and denote it by 
$ (b,a) \mapsto [b,a]. $
  By construction:
 \begin{equation}\label{eq:bracket} [b,a] =  \left. \frac{\partial }{\partial u } 
\right|_{u =0}   \Ad_{\Sigma_u} a,  \end{equation}
with $a,b \in \Gamma(A)$ and $ \Sigma_\epsilon$ an arbitrary smooth family of bisections 
as in Lemma \ref{lem:existsbisections}, i.e. 
$ \Sigma_0=\epsilon(M)$ and $\left.\frac{\partial}{\partial u} \right|_{u =0} \sigma_u(m) = b_m$.

We are now able to make the statement of Theorem \ref{thm:infinitesimal} more precise.

\begin{theo}\label{thm:infinitesimal2}
For every Lie rackoid $\Gamma$, the vector bundle defined in (\ref{eq:fibre}) equipped with
the anchor map defined in (\ref{eq:ancre}) 
and the bracket defined in (\ref{eq:bracket}) is a Leibniz algebroid, called the tangent 
Leibniz algebroid of $\Gamma$. 
\end{theo}
\begin{proof}
We show the compatibility of the bracket with the multiplication by smooth functions 
$f\in{\mathcal C}^{\infty}(M)$.
Relation (\ref{eq:bracket}), applied to a smooth family $(\Sigma_u)_{u \in I}$ of bisections 
corresponding to an arbitrary section $b \in \Gamma(A)$ as in Lemma \ref{lem:existsbisections}, 
implies that for every function~$f$:
 \begin{eqnarray*} [b,fa]  &=& \left. \frac{\partial }{\partial u } \right|_{u =0}   \Ad_{\Sigma_u} (fa) \\
 &=&  \left. \frac{\partial }{\partial u } \right|_{u =0} \left( (\underline{\sigma_u}^{-1})^* f 
\Ad_{\Sigma_u} (a)\right) \\
&=&f \left. \frac{\partial }{\partial u } \right|_{u =0}  \Ad_{\Sigma_u} (a) + \left. 
\frac{\partial }{\partial u } \right|_{u =0} \big( (\underline{\sigma_u}^{-1})^* f \big)\, a, \\
& =&f [b,a]+  \left. \frac{\partial }{\partial u } \right|_{u =0} \left( (\underline{\sigma_u}^{-1})^* 
f \right)\, a. 
 \end{eqnarray*}
where we used (\ref{eq:donnera_ancre}) to go from the first to the second line. 
The proof now follows from the fact that the vector field obtained by differentiating at 
$u=0$ the $1$-parameter family of diffeomorphisms
$ \underline{\sigma_u} := t \circ \sigma_u $ has a value at $m \in M$ given by  
$$Tt  \left( \left. \frac{\partial }{\partial u } \right|_{u =0}(\sigma_u  (m))\right) = 
Tt ( b_m ) =- \rho(b_m)$$
where $\sigma_u$ is, for all $u \in I$, the section of $s$ corresponding to $\Sigma_u$ (so  
that $ \underline{\sigma_u} = t \circ \sigma_u $) and where
Lemma \ref{lem:existsbisections} was used the definition of the anchor given by equation 
(\ref{eq:ancre}) in the last line. This implies that the vector field obtained by differentiating at 
$u=0$ the $1$-parameter family of diffeomorphisms
$ \underline{\sigma_u}^{-1} $ is $\rho(b)$ and that
$$ \left. \frac{\partial }{\partial u } \right|_{u =0} \left( (\underline{\sigma_u}^{-1})^*f \right)= \rho(b)(f) ,$$
which completes the proof of (\ref{eq:anchor}).

Differentiating (\ref{eq:donnera_Leibniz}) with respect  to $\Tau$ at the bisection 
$\epsilon(M)$ yields, in view of Lemma \ref{lem:adjointAsDifferential}, that, for every 
section $b \in \Gamma(A)$:
 $$  \Ad_\Sigma([b,a]) = [\Ad_\Sigma b , \Ad_\Sigma a].$$
Lemma \ref{lem:smooth} implies the continuity of the Leibniz algebroid bracket, which allows 
to differentiate
the previous expression with respect to $\Sigma$ at the point $1_M$ in the direction of some 
$c \in \Gamma(A) \simeq T_{1_M}\Bis(\Gamma) $.
The relation obtained  by this procedure is precisely:
 $$  [c,[b,a]] = [[c, b] , a]+[b,[c,a]].$$
 This completes the proof.
\end{proof}

\section{More Examples}

\subsection{A Lie rackoid integrating a hemisemidirect product Leibniz algebra}
\label{hemisemidirectproduct_Lie_rackoid}

In this section, we present an example of Lie rackoids that comes neither from a Lie groupoid 
nor from a bundle of Lie racks. We will see that this Lie rackoid integrates the 
hemisemidirect product Leibniz algebra associated to the action of diffeomorphisms on 
$1$-forms. For the notion of a hemisemidirect product Leibniz algebra, see \cite{KinWei}. 

Let $M$ be an arbitrary manifold, we define a smooth precategory with base 
manifold $M$ as follows: 
$\Gamma := T^* M \times M$, while $ s,t : \Gamma \to M$ are respectively, the maps,
  $$ t(\alpha, n) = n \,\,\,\hbox{ and  }\,\,\, s(\alpha, n) = m $$
   for all $m \in M$, $\alpha \in T^*_m M $, $n \in M $, and $ 1_m = (0_m,m)$
   with $0_m$ the zero element in $T_m^*M$.
   
   We first characterize its bisections:
   
   \begin{lem}\label{lem:bisects_rackoids_not_groupoids}
   A bisection in ${\rm Bis}(\Gamma) $ of 
$\xymatrix{\Gamma\ar@<2pt>[r]^{s}\ar@<-2pt>[r]_{t} & M}$ is of the form 
$(\omega_m, \varphi(m)) $ with $ \omega \in  \Omega^1 (M)$ a $1$-form and 
$\varphi: M \to M $ a diffeomorphism.
   \end{lem}
   
   This lemma allows us  to see elements of ${\rm Bis}(\Gamma)$ as pairs $(\omega,\varphi) $ 
   with $\varphi \in Diff(M) $ and $\omega \in  \Omega^1 (M) $, an identification that 
we will use without 
   further mention.
   
    We now endow the precategory 
$\xymatrix{\Gamma\ar@<2pt>[r]^{s}\ar@<-2pt>[r]_{t} & M}$ with a 
rackoid structure as follows.
For an arbitrary element $ (\alpha,n) \in \Gamma$ with 
$ \alpha \in T^*_m M, n \in M$, and for an arbitrary
 $(\omega,\varphi) \in {\rm Bis}(M) $, we set:
    \begin{equation}\label{eq:examples_rackoids_not_groupoids} 
(\omega,\varphi)    \rhd (  \alpha , m ) := ((\varphi^{-1})^* \alpha , \varphi(n) )
    \end{equation}
with the understanding that $ \psi^* =  (T_{\psi^{-1}(m)}^*\psi)$. Observe that 
$$T_{(\varphi\circ\psi)^{-1}(m)}^*(\varphi\circ\psi)\,=\,T_{\psi^{-1}(\varphi^{-1}(m))}^*\psi\,\circ\,
T_{\varphi^{-1}(m)}^*\varphi,$$
which we write in a short hand notation as $(\varphi\circ\psi)^*=\psi^*\circ\varphi^*$.
   
   Our first claim is that the product above defines a Lie rackoid structure.
   
   \begin{prop}
   Let $M$ be a manifold, then the operation $\rhd$ defined in equation 
$(\ref{eq:examples_rackoids_not_groupoids})$ is a Lie rackoid structure on 
$\xymatrix{\Gamma\ar@<2pt>[r]^{s}\ar@<-2pt>[r]_{t} & M}$.
   \end{prop}
   
   \begin{prf}
   First, we have to determine what the product $\rhd$ induces on bisections: 
It follows directly 
   from (\ref{eq:examples_rackoids_not_groupoids}) that for every bisection 
$(\omega,\varphi) \in 
{\rm Bis}(\Gamma)$ and $(\eta,\psi) \in {\rm Bis} (\Gamma)$, we have 
         \begin{equation}\label{eq:examples_rackoids_not_groupoids_onbisections}
         (\eta,\psi) \rhd (\omega,\varphi) := 
((\psi^{-1})^* \omega,\psi \circ  \varphi \circ  \psi^{-1}  )
         \end{equation}
    As a consequence, for every pair 
       $(\omega,\varphi) \in {\rm Bis}(\Gamma)$ and $(\eta,\psi) \in {\rm Bis} (\Gamma)$
       and every $(\alpha,n)  \in \Gamma$, we compute:
     \begin{eqnarray*} &  \big((\eta,\psi) \rhd (\omega,\varphi) \big) \rhd 
\big(  (\eta,\psi) 
\rhd (\alpha,n)  \big) & \\ 
=&     ((\psi^{-1})^* \omega ,\psi \circ  \varphi \circ  \psi^{-1})  \rhd 
\big(  (\eta,\psi) \rhd 
(\alpha,n)  \big)   &
      \hbox{by  (\ref{eq:examples_rackoids_not_groupoids_onbisections}) }\\  
 =&      ((\psi^{-1})^* \omega, \psi \circ  \varphi \circ  \psi^{-1})  \rhd 
\big(  (\psi^{-1})^* \alpha, 
\psi(n)  
\big)   &
      \hbox{by  (\ref{eq:examples_rackoids_not_groupoids}) } \\
      =& \big(( ( \psi \circ  \varphi \circ  \psi^{-1})^{-1})^* (\psi^{-1})^* 
\alpha, (\psi \circ  \varphi \circ  
\psi^{-1}) \circ   \psi(n) \big)   &
      \hbox{by  (\ref{eq:examples_rackoids_not_groupoids}) } \\
      =& \big( ((\psi^{-1})^* \circ  (\varphi^{-1})^*)\alpha, (\psi \circ  
\varphi)(n) \big) & \\
      =&(\psi, \eta) \rhd ((\varphi,\omega) \rhd (\alpha,n))  
    & \hbox{by (\ref{eq:examples_rackoids_not_groupoids})} 
     \end{eqnarray*}
   \end{prf}
   
   We now compute the Leibniz algebroid associated with it:
   
   \begin{prop}
   Let $M$ be a manifold, then the tangent Leibniz algebroid of the above Lie rackoid structure on 
$\xymatrix{\Gamma\ar@<2pt>[r]^{s}\ar@<-2pt>[r]_{t} & M}$ is the vector bundle
   $ TM \oplus T^*M$, equipped with the projection onto the first component 
as anchor and the bracket:
    $$ [X+\alpha,Y+\beta] = [X,Y]+ {\mathcal L}_X \beta$$
    for all $X,Y \in {\mathcal X}(M), \alpha, \beta \in \Omega^1(M)$.
   \end{prop}
   \begin{prf}
   This formula can be indeed derived quite easily by differentiating 
   (\ref{eq:examples_rackoids_not_groupoids_onbisections}), as in 
	Theorem \ref{thm:infinitesimal}.
   \end{prf}

Observe that the formula for the Leibniz bracket is exactly the Leibniz bracket of
the hemisemidirect product associated to a Lie algebra (here the Lie algebra of vector fields)
and a module (here the module of $1$-forms), see Example 2.2 (p. 529) in \cite{KinWei}.    

\subsection{Augmented Lie rackoids}

The construction mechanism from Section \ref{augmented_rackoids}
generalizes to the smooth setting. The following theorem also exists in a non-unital
version which we will not spell out explicitly.  

\begin{theo}  \label{augmented_Lie_rackoids}
Let $\xymatrix{\Gamma\ar@<2pt>[r]^{s'}
\ar@<-2pt>[r]_{t'} & M}$ be a Lie groupoid and $\xymatrix{X\ar@<2pt>[r]^{s}
\ar@<-2pt>[r]_{t} & M}$ be a smooth precategory. Suppose that there exists a 
smooth map $p:X\to \Gamma$ such that:
$$\xymatrix{X \ar@<2pt>[d]^{s}\ar@<-2pt>[d]_{t} \ar[r]^p & 
\Gamma \ar@<2pt>[d]^{s'} \ar@<-2pt>[d]_{t'}  \\
M \ar[r]^{\rm id_M} & M }$$
Suppose further that we have a commutative diagram for the identity maps $\epsilon:M\to X$ and
$\epsilon':M\to\Gamma$:
$$\xymatrix{ X \ar[r]^p & \Gamma \\
             M \ar[u]^{\epsilon} \ar[r]^{\rm id_M} & M \ar[u]^{\epsilon'} }$$ 

Suppose that the Lie groupoid ${\rm Bis}(\Gamma)$ of  bisections acts on $X$
smoothly and that for all $g\in{\rm Bis}(\Gamma)$ and all $x\in X$: 
$$p(g\cdot x)\,=\,g\,p(x)\,g^{-1},$$
and that $p(x)\cdot 1_m=1_m$ for all $x\in X$ and all $m\in M$.

Then the prescription 
$$x\rhd y\,:=\,p(x)\cdot y$$
defines for $y\in X$ and a  bisection $x$ a Lie rackoid structure on $X$. 
\end{theo}

As examples for the preceding theorem, we can consider associated tensor bundles to the pair 
Lie groupoid $\xymatrix{\Gamma\ar@<2pt>[r]^{s'}\ar@<-2pt>[r]_{t'} & M}$: 

\begin{cor}
Let $\Gamma$ be a Lie groupoid over $M$.
For every $ q\in {\mathbb N}$, 
$\bigotimes^q T\Gamma$, $\Lambda^q T\Gamma$, $\bigotimes^q T^*\Gamma$ or 
$\Lambda^q T^*\Gamma$ admit natural Lie rackoid structures over $M$. 
\end{cor}

Let us discuss another application of Theorem \ref{augmented_Lie_rackoids}. 

Let $M$ be a manifold. 
Take the Lie groupoid $\xymatrix{T^*M\oplus T^*M\ar@<2pt>[r]^{s}\ar@<-2pt>[r]_{t} & M}$. 
Its bisections are (like in Lemma \ref{lem:bisects_rackoids_not_groupoids})
of the form $(\omega_m,\alpha_n,\phi)$ where $\omega_m\in T^*_mM$, $\alpha_n\in T^*_nM$
and $\phi$ is a diffeomorphism between open sets, sending $m$ to $\phi(m)=n$. Take as the map 
$p$ the forgetful map $p:T^*M\oplus T^*M\to M\times M$ with values in the pair groupoid
$\xymatrix{M\times M\ar@<2pt>[r]^{s'}\ar@<-2pt>[r]_{t'} & M}$. The bisections of the pair 
groupoid are diffeomorphisms $\psi$ on $M$. These bisections act on $T^*M\oplus T^*M$ in 
a natural way, by composing diffeomorphism and acting via $\psi^*$ on cotangent vectors. 
It is easy to see that the $p$-image of the action of $\psi$ on $(\omega_m,\alpha_n,\phi)$
is the conjugation of $\phi$ by $\psi$. Therefore we are in position to apply the above
theorem in order to obtain a rackoid structure on $T^*M\oplus T^*M$. This is another way to
obtain the Lie rackoid which we described in Section \ref{hemisemidirectproduct_Lie_rackoid}:

\begin{cor}
The Lie rackoid constructed in Section \ref{hemisemidirectproduct_Lie_rackoid}
can be obtained as a special case of Theorem \ref{augmented_Lie_rackoids}
using the forgetful map $p$ from the Lie groupoid $T^*M\oplus T^*M$ to the 
pair groupoid $M\times M$. 
\end{cor}

Let us also come back to a smooth version of the construction in 
Section \ref{augmented_rackoids}.

\begin{cor}
Let $q:Y\to M$ be a smooth fiber bundle such that a Lie groupoid 
$\xymatrix{\Gamma\ar@<2pt>[r]^{s'}\ar@<-2pt>[r]_{t'} & M}$ acts freely and 
transitively on the fibers of $q$. Then the fiber product $X:=Y\times_{q,M,q}Y$ 
becomes naturally a non-unitary Lie rackoid. 
\end{cor} 

\subsection{The fundamental rackoid}
\label{sec:link}

Here we explain a rackoid version of the fundamental rack defined e.g. in \cite{FenRou}
p.358. In order to be close to their construction, we will work here with right racks instead of 
left racks, and consequently with right rackoids instead of left rackoids. 
 
A {\it link} is a codimension two embedding $L:M\subset Q$ of manifolds. We will 
assume that $M$ is 
non-empty, that $Q$ is connected (with empty boundary) and that $M$ is transversely 
oriented in $Q$.
In other words, we assume that each normal disc to $M$ in $Q$ has an orientation 
which is locally
and globally coherent. 

The link is called {\it framed} if there is a cross-section $\lambda:M\to\partial N(M)$ 
of the 
normal disk bundle. Denote by $M^+$ the image of $M$ under $\lambda$. 
In the following, we will only
consider framed links.   

Then, Fenn and Rourke associate to $L\subset Q$ an augmented rack, called the 
{\it fundamental rack} of the link
$L$, which is the space $\Gamma$ of homotopy classes of paths in 
$Q_0:={\rm closure}(Q\setminus N(L))$ of $L$, from a point in $M^+$ to some base point $q_0$. 
During the homotopy, the final point of the path at $q_0$ is kept fixed and the initial
point is allowed to wander at will on $M^+$.

The set $\Gamma$ has an action of the fundamental group $\pi_1(Q_0,q_0)$ defined as follows: 
let $\gamma$ 
be a loop in $Q_0$ based at $q_0$ representing an element $g\in\pi_1(Q_0)$. 
If $\alpha\in\Gamma$ is represented by the 
path $\alpha$, define $a\cdot g$ to be the class of the composite path $\alpha\circ\gamma$.

We can use this action to define a rack structure on $\Gamma$. Let $p\in M^+$ be a point on the 
framing image. Then $p$ lies on an unique meridian circle  of the normal disc bundle. 
Let $m_p$ be the
loop based at $p$ which follows the meridian around in a positive direction. 
Let $a,b\in\Gamma$ be
represented by paths $\alpha,\beta$ respectively. 
Let $\partial(b)$ be the element of $\pi_1(Q_0,q_0)$
determined by the homotopy class of the loop $\beta^{-1}\circ m_{\beta(0)}\circ\beta$. The 
\emph{fundamental rack of the framed link $L$} is defined to be the set $\Gamma=\Gamma(L)$ 
with the operation
$$a\lhd b\,:=\,a\cdot\partial(b):\,=\,[\alpha\circ\beta^{-1}\circ m_{\beta(0)}\circ\beta].$$  

Fenn and Rourke show in \cite{FenRou} Proposition 3.1, p. 359, that $\Gamma$ is indeed a rack,
and go on to show that $\partial:\Gamma\to\pi_1(Q_0,q_0)$ is an augmented rack.

Let us now come to an "-oidification" of their construction, i.e. the construction of a 
fundamental rackoid $\Phi$ on which the fundamental groupoid $\Pi(Q_0)$, 
i.e. the set of homotopy classes of paths in $Q_0$ from some point $q_0$ to 
some other point $q_1$, acts naturally.

The main idea is to replace the set of homotopy classes $\Gamma$ by the set 
$\Phi(Q_0)$ consisting of homotopy classes of paths $\alpha$ from some point $s(\alpha)=q_0\in Q_0$ to a  
point in $M^+$ at $\frac{1}{2}$, and then back to some other point $t(\alpha)=q_1\in Q_0$. There is
a natural action of the fundamental groupoid $\Pi(Q_0)$
on $\Phi(Q_0)$ from the left resp. from the right induced by concatenation of paths. 
As before one has to rescale in order to keep the passage in $M^+$ at $t=\frac{1}{2}$. 

There is also a natural map $\partial:\Phi(Q_0)\to\Pi(Q_0)$ where for 
$\alpha\in\Phi(Q_0)$ with $s(\alpha)=q_0$ and $t(\alpha)=q_1$, 
$\partial\alpha$ is given by the concatenation of the 
paths from $q_0$ to $\alpha(\frac{1}{2})$, the meridian at the point $\alpha(\frac{1}{2})$, 
and finally the paths from $\alpha(\frac{1}{2})$
to $q_1$. In the same way as before using the non-unitary version of 
Theorem \ref{thm:augmented_rackoids}, 
this map gives rise to an non-unital augmented rackoid, which we call
the {\it fundamental rackoid} of the link $L$.

\section{Conclusion}

Having established the basic properties of Lie rackoids, we are left with many 
open questions. 

First of all, there is structure theory for Lie rackoids to be done (morphisms,
weak morphisms, i.e. Morita equivalences, etc). The functoriality of the concept is 
a important feature for modern differential geometry. 

Concerning Lie rack bundles, one may ask about the notion of principal $R$-bundles 
where $R$ is a Lie rack. It should not be too difficult to find a suitable definition and 
to show that these are special cases of Lie rackoids. 
As for principal bundles, one would like to have Atiyah sequences for those.
  
One of the most important questions is certainly about the integration of Leibniz 
algebroids into Lie rackoids. For the moment, there are several constructions 
integrating finite-dimensional Leibniz algebras to Lie racks, but unfortunately not in 
a functorial way. One could try to integrate Leibniz algebroids along the
lines of Mackenzie for transitive Leibniz algebroids, or right away along the lines 
of Crainic-Fernandes for arbitrary Leibniz algebroids, with probably some discreteness 
conditions coming into the game (for rendering the analogue of the Weinstein 
${\mathcal C}^0$-manifold a smooth manifold).  

Lie rackoids seem to be the ``integral manifolds'' which correspond to some kind of 
non-commutative foliations (the Leibniz algebroid). Observe that we did not admit 
non-Hausdorff manifolds in our setting, excluding interesting examples from foliation 
theory. Even more fundamental, in the same way as
\'etale Lie groupoids may be seen as generalized spaces (see \cite{CraMoe}),
\'etale Lie rackoids may be seen as generalized non-commutative
spaces. The open problem is to make this relation precise in whatever 
non-commutative framework you prefer.

\end{document}